\documentclass[11pt]{amsart}

\usepackage{AFBoix2015}

\begin{document}

\author[A.\,F.\,Boix]{Alberto F.\,Boix}
\address{Department of Mathematics, Universitat Polit\`ecnica de Catalunya BarcelonaTech, Av. Eduard Maristany 16, 08019, Barcelona, Spain.}
\email{alberto.fernandez.boix@upc.edu}

\author[B.\,K.\,Lima--Pereira]{B\'arbara K.\,Lima--Pereira}
\address{Universidade Federal de Lavras - C\^{a}mpus S\~{a}o Sebasti\~{a}o do Para\'iso - MG, Rua Ant\^{o}nio Carlos Pinheiro de Alc\^{a}ntara, 855 - Jardim Mediterran\'{e} - CEP: 37950-000, Minas Gerais, Brazil.}
\email{barbarapereira@ufla.br}



\title[A local variant of the BEH conjecture]{Exploring a local variant of the Buchsbaum--Eisenbud--Horrocks conjecture}

\begin{abstract}
The Buchsbaum--Eisenbud--Horrocks conjecture has attracted the attention of many researchers working in Commutative Algebra and Algebraic Geometry in the last fifty years. Quite recently, a variant of this conjecture has been formulated by Lima--Pereira, Nu{\~n}o--Ballesteros, Orefice--Okamoto and Tomazella in their study of complete intersection singularities. The purpose of this paper is, on the one hand, to exhibit some cases where this conjecture holds. On the other hand, we show that the conjecture fails in general by exhibiting some counterexamples. Finally, we formulate a conjecture that can be regarded as a local variant of the Buchsbaum--Eisenbud--Horrocks' one.
\end{abstract}


\maketitle

\section*{Introduction}
The Buchsbaum--Eisenbud--Horrocks conjecture, formulated at the end of the $80$'s independently by Buchsbaum and Eisenbud \cite[1.4]{BuchsbaumEisenbudstructurefreeres} and Horrocks \cite[Problem 24]{BEHhorrocks}, can be stated in the following way.

\begin{con*}[Buchsbaum--Eisenbud--Horrocks]
Let $R$ be a commutative Noetherian ring such that $\operatorname{Spec}(R)$ is connected, and let $M$ be a non--zero, finitely generated $R$--module of finite projective dimension. Then, for any finite projective resolution of $M$
\[
\xymatrix{0\ar[r]& P_d\ar[r]& \ldots\ar[r]& P_1\ar[r]& P_0\ar[r]& M\ar[r]& 0,}
\]
we have
\[
\rank_R (P_i)\geq\binom{c}{i},
\]
where $c=\height ((0:_R M))$.
\end{con*}

The interested reader in the Buchsbaum--Eisenbud--Horrocks conjecture may like to consult \cite{BoocherGrifosurveyBEH}, a nice survey of what is known and unknown about this problem, especially when $R$ is a polynomial ring in several variables over a field. We also want to mention that in \cite[Conjecture 4.8 and Lemma 4.9]{BoocherGrifosurveyBEH}, it is observed that, in order to prove the Buchsbaum--Eisenbud--Horrocks problem in the case where $R$ is a local ring, it is enough to restrict the attention to the class of modules with finite length.

The purpose of this paper is to study the following conjecture, formulated by Lima--Pereira, Nu{\~n}o--Ballesteros, Orefice--Okamoto and Tomazella in their study of complete intersection singularities (see \cite[Conjecture 4.9]{conjectureTor}). It may be regarded as a variant of the Buchsbaum--Eisenbud--Horrocks conjecture in the local case.

\begin{con}[Lima--Pereira, Nu{\~n}o--Ballesteros, Orefice--Okamoto, Tomazella]\label{conjecture length Tor}
Let $(R,\mathfrak{m})$ be a regular local ring of dimension $n$, let $I$ be an ideal generated by an $R$--regular sequence of length $k\leq n,$ and let $J$ be an ideal generated by a regular sequence of length $n.$ Assume, in addition, that $I+J$ is $\mathfrak{m}$--primary. Then, if $\ell$ denotes the length function in the category of $R$--modules, then
\[
\ell\left(\Tor_i^R (R/I,R/J)\right)=\binom{k}{i}\ell\left(\frac{R}{I+J}\right)
\]
for any $0\leq i\leq k.$
\end{con}
We have two main goals in this paper: on the one hand, we give particular settings where Conjecture \ref{conjecture length Tor} holds. On the other hand, we show that it is in general not true by exhibiting some counterexamples. At the end, we formulate a more general conjecture, exhibiting some particular examples where it holds.

Now, let us provide a more detailed overview of the contents of this paper for the convenience of the reader. In Section \ref{section: some positive results}, we present some general settings where Conjecture \ref{conjecture length Tor} holds. More precisely, we can phrase the main results obtained in Section \ref{section: some positive results} as follows.

\begin{teo}\label{summary of main positive results}
Conjecture \ref{conjecture length Tor} holds if at least one of the following assumptions holds.

\begin{enumerate}[(i)]

\item\ Either $I\subseteq J$ or $J\subseteq I$ (see Proposition \ref{the conjecture when one has some inclusion}).

\item\ If $k=2$ (see Discussion \ref{alternative proof of the conjecture for k equal two}).

\item\ $J=(f_1,\ldots,f_n)$, $(f_1,\ldots,f_k)\subseteq I$ for some $1\leq k\leq n-1$, and $f_{k+1},\ldots,f_n$ is an $(R/I)$--regular sequence (see Proposition \ref{the conjecture in good position}).

\item\ $n\geq 4$, $J=(f_1,\ldots,f_n)$, $(f_1,\ldots,f_{k-1})\subseteq I$ for some $3\leq k\leq n-1$, and $f_{k+1},\ldots,f_n$ is an $(R/I)$--regular sequence (see Proposition \ref{the conjecture in quasi good position}).

\end{enumerate}
\end{teo}
Some remarks are in order here. First, part (i) of Theorem \ref{summary of main positive results} was already proved by Borna \cite{Borna2008} in his Ph.D. thesis. Second, part (ii) of Theorem \ref{summary of main positive results} was proved in \cite{conjectureTor}. However, in this paper we present alternative proofs of these cases, which allow us to obtain more general statements, as the ones obtained in parts (iii) and (iv) of Theorem \ref{summary of main positive results}.

In Section \ref{the monomial case section}, we will show that Conjecture \ref{conjecture length Tor} always holds for a certain subclass of complete intersection monomial ideals (see Notation \ref{notation in the monomial setting}). In Section \ref{section: some counterexamples}, we exhibit explicit counterexamples to Conjecture \ref{conjecture length Tor} showing, in particular, that this conjecture cannot hold for any complete intersection binomial ideal.

Section \ref{section: singularity theory} can be independently read from the rest of the paper. In this section, we briefly review the motivation for studying Conjecture \ref{conjecture length Tor} from the perspective of Singularity Theory. Finally, in Section \ref{more general conjecture section}, we formulate a more general conjecture, that we feel is the real generalization of the Buchsbaum--Eisenbud--Horrocks' one in this setting. We also raise the corresponding Total Rank Conjecture in this local framework, showing that it cannot be deduced directly from the celebrated papers by VandeBogert and Walker \cite{WalkerTRCodd,WalkerTRCeven} proving the classical Total Rank Conjecture.


Throughout this paper, all the rings are commutative with unity. Moreover, given a commutative ring $R$, $\ell_R$ (or simply $\ell$ if the context is clear) will always denote the length function in the category of $R$--modules. Finally, by a local ring $(R,\mathfrak{m})$ we mean a commutative Noetherian ring $R$ with a unique maximal ideal $\mathfrak{m}$. 

\section{Some positive general results}\label{section: some positive results}

The goal of this section is to provide several situations where Conjecture \ref{conjecture length Tor} has a positive answer. We start by reviewing an elementary fact about the behavior of the length function for the sake of completeness, see \cite[page 15, 1.2.25]{BrunsHerzog1993} and \cite[\href{https://stacks.math.columbia.edu/tag/02M1}{Tag 02M1}]{stacks-project} for a proof.

\begin{lm}\label{length and flat base change}
Let $\xymatrix@1{(R,\mathfrak{m}_R)\ar[r]^-{\phi}& (S,\mathfrak{m}_S)}$ be a homomorphism of local rings, and let $N$ be an $S$--module which is flat as $R$--module ($N$ is regarded as $R$--module by restriction of scalars under $\phi$) such that $\ell_S (N/\mathfrak{m}_R N)<\infty$. Then, for any $R$--module $M$ of finite length, we have
\[
\ell_S (M\otimes_R N)=\ell_R (M)\cdot\ell_S\left(\frac{N}{\mathfrak{m}_R N}\right).
\]
\end{lm}
The next thing we plan to show is that there is no loss of generality in assuming in Conjecture \ref{conjecture length Tor} that $R$ is a complete regular local ring. This is a consequence of the following technical statement.

\begin{prop}\label{conjecture tor length and base change}
Let $(R,\mathfrak{m})$ be a local ring of dimension $n$, and let $\xymatrix@1{R\ar[r]^-{\phi}& S}$ be a flat local ring homomorphism, with $(S,\mathfrak{m}_S)$ a regular local ring, such that $\dim (S)=\dim (R)$ and that $\ell_S (S/\mathfrak{m}S)<\infty$. Assume that Conjecture \ref{conjecture length Tor} holds for $S$. Then, Conjecture \ref{conjecture length Tor} is also valid for $R$.
\end{prop}

\begin{proof}
\ First, we observe that, since $S$ is regular and flat over $R$ we have, by \cite[page 162, Theorem 51]{Matsumuratex}, that $R$ is also regular. Now, let $I$ be an ideal of $R$, generated by a regular sequence of length $k\leq n,$ and let $J$ be an ideal of $R$ generated by a regular sequence of length $n.$ Assume, in addition, that $I+J$ is $\mathfrak{m}$--primary. By \cite[Proposition 1.1.2]{BrunsHerzog1993}, we know that $IS$ and $JS$ are ideals of $S$ generated by regular sequences of lengths $k$ and $n$ respectively. Moreover, since $I+J$ is $\mathfrak{m}$--primary, we also have, since by assumption $\phi$ is flat with closed fiber of finite length, that $(I+J)S=IS+JS$ is also $\mathfrak{m}_S$--primary because of Lemma \ref{length and flat base change}. In this way, since Conjecture \ref{conjecture length Tor} holds for $S$ we have
\[
\ell_S\left(\Tor_i^S (S/IS,S/JS)\right)=\binom{k}{i}\ell_S\left(\frac{S}{(I+J)S}\right).
\]
Notice that this equality is equivalent to the following one, because $\phi$ is flat \cite[Corollary 10.61]{Rotman2009}.
\begin{equation}\label{length from S to R}
\ell_S\left(\Tor_i^R (R/I,R/J)\otimes_R S\right)=\binom{k}{i}\ell_S\left(\frac{R}{I+J}\otimes_R S\right).
\end{equation}
Now, on the one hand we observe that, by assumption, $\ell_S (S/\mathfrak{m}S)<\infty$. On the other hand, since $I+J$ is $\mathfrak{m}$--primary, and $\Tor_i^R (R/I,R/J)$ is a finitely generated $R$--module that is annihilated by $I+J$, we have, by \cite[Tag 00J0]{stacks-project}, that $\Tor_i^R (R/I,R/J)$ has finite length. In this way, keeping in mind these facts, we can use Lemma \ref{length and flat base change} again to rewrite \eqref{length from S to R} as follows.
\[
\ell_R\left(\Tor_i^R (R/I,R/J)\right)\cdot\ell_S(S/\mathfrak{m}S)=\binom{k}{i}\ell_R\left(\frac{R}{I+J}\right)\cdot\ell_S (S/\mathfrak{m}S).
\]
And this equality is equivalent, using again our assumption that $\ell_S (S/\mathfrak{m}S)<\infty$, to
\[
\ell_R\left(\Tor_i^R (R/I,R/J)\right)=\binom{k}{i}\ell\left(\frac{R}{I+J}\right),
\]
which is precisely what we want to prove.
\end{proof}

\begin{cor}\label{reduction to the complete case}
In Conjecture \ref{conjecture length Tor} we can assume, without loss of generality, that $(R,\mathfrak{m})$ is a regular local ring, complete with respect to the $\mathfrak{m}$--adic topology. Equivalently, by Cohen structure theorem \cite[pages 217--218]{Matsumuratex}, we can assume that $R$ is a formal power series ring with coefficients either on a field or over a Cohen ring.
\end{cor}
We also need a graded variant of Proposition \ref{conjecture tor length and base change}, which we formulate now. The following statement will be useful for us along Section \ref{the monomial case section}.

\begin{prop}\label{graded conjecture tor length and base change}
Let $\mathbb{K}$ be a field, let $S:=\mathbb{K}[x_1,\ldots,x_n]$ be the polynomial ring in $n$ variables over $\mathbb{K}$, regarded as graded ring with its natural standard $\mathbb{Z}$--grading. Let $I,J\subset S$ be ideals generated by regular sequences of length $k$ and $n$ respectively such that both regular sequences are made up by homogeneous elements, and set $\mathfrak{m}:=(x_1,\ldots,x_n)$. Moreover, assume that $\sqrt{I+J}=\mathfrak{m}$. Then, Conjecture \ref{conjecture length Tor} holds for $S$, $I$ and $J$ if and only if it holds for $R:=S_{\mathfrak{m}}$, $IR$ and $JR$.    
\end{prop}

\begin{proof}
\ Once again, by \cite[Proposition 1.1.2]{BrunsHerzog1993}, $IR$ and $JR$ are also generated by regular sequences in $R$ of lengths $k$ and $n$ respectively. Moreover, by \cite[Proposition 1.5.15]{BrunsHerzog1993}, $-\otimes_S R$ is faithfully flat in the category of graded $S$--modules. Since both $\Tor_i^S (S/I,S/J)$ and $S/(I+J)$ are finite length, graded $S$--modules supported at $\mathfrak{m}$, and for these modules localization at $\mathfrak{m}$ preserves length, we have that
\[
\ell_S\left(\Tor_i^S (S/I,S/J)\right)=\binom{k}{i}\ell\left(\frac{S}{I+J}\right)
\]
if and only if
\[
\ell_R\left(\Tor_i^R (R/IR,R/JR)\right)=\binom{k}{i}\ell\left(\frac{R}{(I+J)R}\right),
\]
just what we want to prove.
\end{proof}

As already pointed out in \cite{conjectureTor}, a particular case of Conjecture \ref{conjecture length Tor} was proved in \cite[Lemma 3.4.1]{Borna2008}; in the next result we reproduce the proof presented there for the convenience of the reader.

\begin{prop}\label{the conjecture when one has some inclusion}
Conjecture \ref{conjecture length Tor} holds if either $I\subseteq J$ or $J\subseteq I.$
\end{prop}

\begin{proof}
\ First, assume that $I\subseteq J$. Let $f_1,\ldots,f_k$ be an $R$--regular sequence generating $I,$ and let $K_{\bullet} (f_1,\ldots,f_k;R)$ be the Koszul complex with respect to this sequence of elements. Since $f_1,\ldots, f_k$ is an $R$--regular sequence, $K_{\bullet} (f_1,\ldots,f_k;R)$ is a free resolution of $R/I$ of length $k.$ Therefore, one has that
\[
\Tor_i^R (R/I,R/J)\cong H_i \left(K_{\bullet} (f_1,\ldots,f_k;R)\otimes_R R/J\right).
\]
However, since by assumption $I\subseteq J$ one has that all the differentials of $K_{\bullet} (f_1,\ldots,f_k;R)\otimes_R R/J$ are zero, and therefore one ends up, for each $i\geq 0$, with an $R/J$--isomorphism
\[
\Tor_i^R (R/I,R/J)\cong \left(R/J\right)^{\oplus\binom{k}{i}}.
\]
Since the length function is additive and $I\subseteq J,$ one obtains that
\[
\ell\left(\Tor_i^R (R/I,R/J)\right)=\binom{k}{i}\ell\left(R/J\right)=\binom{k}{i}\ell\left(R/(I+J)\right),
\]
as claimed.

On the other hand, assume that $J\subseteq I$. Since $J$ is generated by a regular sequence of length $n$, we have that $I$ is $\mathfrak{m}$--primary, hence $k=n$. Now, let $g_1,\ldots,g_n$ be an $R$--regular sequence generating $J$, and let $K_{\bullet} (g_1,\ldots,g_n;R)$ be the Koszul complex with respect to this sequence of elements, which is a free resolution of $R/J$ of length $n$. In this way, we have $\Tor_i^R (R/I,R/J)\cong H_i \left(K_{\bullet} (g_1,\ldots,g_n;R)\otimes_R R/I\right)$. Since by assumption $J\subseteq I$ one has that all the differentials of $K_{\bullet} (g_1,\ldots,g_n;R)\otimes_R R/I$ are zero, and therefore one ends up, for each $i\geq 0$, with an $R/I$--isomorphism
\[
\Tor_i^R (R/I,R/J)\cong \left(R/I\right)^{\oplus\binom{n}{i}}.
\]
Since the length function is additive and $J\subseteq I$, one finally obtains that
\[
\ell\left(\Tor_i^R (R/I,R/J)\right)=\binom{n}{i}\ell\left(R/I\right)=\binom{n}{i}\ell\left(R/(I+J)\right).
\]
This concludes the proof.
\end{proof}
Proposition \ref{the conjecture when one has some inclusion} can also be proved using the following result obtained in \cite[Lemma 2.2]{weakcompleteintersectionpaper}.

\begin{lm}\label{on weak complete intersections}
Let $(R,\mathfrak{m})$ be a commutative Noetherian local ring, let $I\subseteq R$ be an ideal, and let $\xymatrix@1{F_{\bullet}\ar[r]& M}$ be a minimal free resolution of a finitely generated $R$--module $M$. Moreover, set
\[
\beta_j (M):=\dim_{R/\mathfrak{m}} \left(\Tor_j^R (R/\mathfrak{m},M)\right),
\]
the $j$--th Betti number of $M$.
Then, the following statements are equivalent:

\begin{enumerate}[(i)]

\item\ Every differential $d_j$ of $F_{\bullet}$ satisfies $d_j (F_j)\subseteq I\cdot F_{j-1}.$

\item\ For any $j\geq 0,$ $\Tor_j^R (M,R/I)$ is a free $R/I$--module of rank $\beta_j (M).$

\end{enumerate}
\end{lm}

The motivation for considering Lemma \ref{on weak complete intersections} stems from the following notion, which was introduced in \cite[Definition 2.1]{weakcompleteintersectionpaper}, see also \cite[Definition 2.4]{Diethorn2020} for a closely related notion.

\begin{df}\label{weak complete intersection definition}
Let $R$ be a commutative Noetherian local ring, and let $I$ be an ideal of $R.$ We say that $I$ is a \textbf{weak complete intersection ideal} if, for any integer $j\geq 0,$ $\Tor_j^R (R/I,R/I)$ is a free $R/I$--module of rank $\beta_j (R/I)$.
\end{df}

Actually, thanks to Lemma \ref{on weak complete intersections}, we can extend Proposition \ref{the conjecture when one has some inclusion} in the following way:

\begin{prop}\label{the conjecture with some inclusion and weak complete intersection}
Let $(R,\mathfrak{m})$ be a commutative Noetherian local ring, let $I$ be an ideal of $R$ with finite projective dimension, and let $J\subseteq R$ be an ideal. Assume that $J$ is $\mathfrak{m}$--primary, that $I\subseteq J,$ and that $I$ is a weak complete intersection ideal. Then, one has that
\[
\ell (\Tor_j^R (R/I,R/J))=\beta_j (R/I)\ell\left(\frac{R}{J}\right).
\]
\end{prop}

\begin{proof}
\ Let $F_{\bullet}$ be the minimal free resolution of $R/I$ as $R$--module; since by assumption $I$ is a weak complete intersection, Lemma \ref{on weak complete intersections} guarantees that all the differentials of $F_{\bullet}$ can be represented by matrices whose entries are in $I.$ Since $I\subseteq J,$ one has that $F_{\bullet}\otimes_R R/J$ is a complex where all its differentials are zero, which implies that
\[
H_j (F_{\bullet}\otimes_R R/J)\cong \left(\frac{R}{J}\right)^{\oplus\beta_j (R/I).}
\]
This isomorphism implies, by the additivity of the length, that
\[
\ell (\Tor_j^R (R/I,R/J))=\ell (H_j (F_{\bullet}\otimes_R R/J))=\beta_j (R/I)\ell\left(\frac{R}{J}\right),
\]
just what we wanted to show.
\end{proof}

The following result is a direct consequence of \cite[Theorem 2.2]{MillerRahmatiStriuliduality}.

\begin{prop}\label{just a half of the Tors}
Preserving the notations of Conjecture \ref{conjecture length Tor}, for any integer $i\geq 0,$ there is an $R/J$--isomorphism
\[
\Hom_{R/J}(\Tor_i^R (R/I,R/J),R/J)\cong\Tor_{k-i}^R (R/I,R/J).
\]
In particular, we get the equality $\ell(\Tor_i^R (R/I,R/J))=\ell(\Tor_{k-i}^R (R/I,R/J)).$

\end{prop}

As an immediate consequence of Proposition \ref{just a half of the Tors} one obtains:

\begin{cor}\label{only half of the lengths needed}
Conjecture \ref{conjecture length Tor} holds if and only if, for any $1\leq i\leq\lfloor k/2\rfloor$, we have
\[
\ell\left(\Tor_i^R (R/I,R/J)\right)=\binom{k}{i}\ell\left(\frac{R}{I+J}\right).
\]
\end{cor}
In particular, Corollary \ref{only half of the lengths needed} leads to a different proof of Conjecture \ref{conjecture length Tor} for the case $k=2$ with respect to the one given in \cite{conjectureTor}.

\begin{disc}\label{alternative proof of the conjecture for k equal two}
Assume that $I=(f_1,f_2),$ where $f_1,f_2$ is an $R$--regular sequence. Set $A:=R/J;$ in this particular case, the complex $K_{\bullet}(f_1,f_2;R)\otimes_R A$ is given by
\[
\xymatrix{0\ar[r]& A\ar[r]^-{d_2}& A^2\ar[r]^-{d_1}& A\ar[r]& 0.}
\]
This complex can be split into the following short exact sequences:
\begin{align*}
& \xymatrix{0\ar[r]& \ker(d_2)=\Tor_2^R (R/I,R/J)\ar[r]& A\ar[r]& \im(d_2)\ar[r]& 0}\\
& \xymatrix{0\ar[r]& \im(d_2)\ar[r]& \ker (d_1)\ar[r]& \Tor_1^R (R/I,R/J)\ar[r]& 0}\\
& \xymatrix{0\ar[r]& \ker(d_1)\ar[r]& A^2\ar[r]& \frac{I+J}{J}\ar[r]& 0}\\
& \xymatrix{0\ar[r]& \frac{I+J}{J}\ar[r]& A\ar[r]& \frac{R}{I+J}\ar[r]& 0.}
\end{align*}
We know, thanks to Corollary \ref{only half of the lengths needed}, that
\[
\ell(\Tor_2^R (R/I,R/J))=\ell\left(\frac{R}{I+J}\right).
\]
Keeping in mind this fact, and using the additivity of the length with short exact sequences one has that
\begin{align*}
& \ell(\Tor_1^R (R/I,R/J))=\ell(\ker(d_1))-\ell(\im(d_2))=2\ell(A)-\ell\left(\frac{I+J}{J}\right)-\ell(\im(d_2))\\
& =2\ell(A)+\ell\left(\frac{R}{I+J}\right)-\ell(A)-\ell(\im(d_2))
=\ell(A)+\ell\left(\frac{R}{I+J}\right)-\ell(\im(d_2))\\
& =\ell\left(\frac{R}{I+J}\right)+\ell(A)-\ell(\im(d_2))=\ell\left(\frac{R}{I+J}\right)+\ell(\Tor_2^R (R/I,R/J))=2\ell\left(\frac{R}{I+J}\right),
\end{align*}
just what we wanted to show.
\end{disc}
Another particular instance where Conjecture \ref{conjecture length Tor} holds is given in the following statement. We present two proofs of this result for the reader's benefit. In the next result, given a sequence of elements $\mathbf{g}=g_1,\ldots,g_t$ in a commutative ring $A$, $H_i (\mathbf{g};A)$ will denote the $i$th homology group of the Koszul complex of $A$ with respect to the sequence $\mathbf{g}$.

\begin{prop}\label{the conjecture in good position}
Suppose that $I$ can be generated by a regular sequence of length $k$, and let $J=(f_1,\ldots,f_n)$, where $\mathbf{f}=f_1,\ldots,f_n$ is an $R$--regular sequence. Assume that $(f_1,\ldots,f_k)\subseteq I$ and that $f_{k+1},\ldots,f_n$ is an $R/I$--regular sequence. Then, Conjecture \ref{conjecture length Tor} holds.
\end{prop}

\begin{proof}
\ Recall that, given $B$ an arbitrary commutative Noetherian ring, and given $\mathbf{b}=b_1,\ldots, b_t$ a sequence of elements of $B,$ by \cite[Remark 1.4]{Hunekekoszulhomology} one has, for any integer $i\geq 0,$ that
\begin{equation}\label{huneke trick}
H_i (\mathbf{b},0;B)\cong H_i (\mathbf{b};B)\oplus H_{i-1} (\mathbf{b};B).
\end{equation}
Now, we come back to our setting; our assumption $(f_1,\ldots,f_k)\subseteq I$ implies that
\[
H_i (\mathbf{f};R/I)\cong H_i (f_{k+1},\ldots,f_n,0,\ldots, 0;R/I).
\]
Therefore, using \eqref{huneke trick} $k$ times one ends up with the following isomorphism:
\[
H_i (\mathbf{f};R/I)\cong\bigoplus_{j=0}^i H_{i-j}(f_{k+1},\ldots,f_n;R/I)^{\oplus\binom{k}{j}}.
\]
However, since by assumption $f_{k+1},\ldots,f_n$ is an $R/I$--regular sequence, the last isomorphism boils down to the below one:
\[
H_i (\mathbf{f};R/I)\cong H_0 (f_{k+1},\ldots,f_n;R/I)^{\oplus\binom{k}{i}}\cong\left(\frac{R}{I+J}\right)^{\oplus\binom{k}{i}}.
\]
Summing up, since $\Tor_i^R (R/I,R/J)\cong H_i (\mathbf{f};R/I)$, one concludes, combining these last two isomorphisms, that
\[
\ell(\Tor_i^R (R/I,R/J))=\binom{k}{i}\ell\left(\frac{R}{I+J}\right),
\]
as claimed.
\end{proof}

\begin{proof}[Alternative proof of Proposition \ref{the conjecture in good position}]
\ Since, by assumption, $f_{k+1},\ldots, f_n$ is an $R/I$--regular sequence and that
\[
\frac{R}{I+(f_{k+1},\ldots,f_n)}=\frac{R}{I+J},
\]
one has, by \cite[Corollary 5.16]{Vasconcelosbookintegralclosure}, that
\[
H_i (\mathbf{f};R/I)\cong H_i (f_1,\ldots,f_k; R/(I+J)).
\]
Finally, since $(f_1,\ldots, f_k)\subseteq I+J$ one has that all the differentials of $K_{\bullet} (f_1,\ldots,f_k; R/(I+J))$ are zero and therefore one finally ends up with an isomorphism of $R/(I+J)$--modules
\[
H_i (f_1,\ldots,f_k; R/(I+J))\cong\left(\frac{R}{I+J}\right)^{\oplus\binom{k}{i}}.
\]
Summing up, we obtain
\[
\ell(\Tor_i^R (R/I,R/J))=\ell(H_i (\mathbf{f};R/I))=\ell(H_i (f_1,\ldots,f_k; R/(I+J)))=\binom{k}{i}\ell\left(\frac{R}{I+J}\right),
\]
as claimed.
\end{proof}

We can extend the conclusion obtained in Proposition \ref{the conjecture in good position} in the following way.

\begin{prop}\label{the conjecture in quasi good position}
Let $k\geq 3$, suppose that $I$ can be generated by a regular sequence of length $k$, and let $J=(f_1,\ldots,f_n),$ where $\mathbf{f}=f_1,\ldots,f_n$ is an $R$--regular sequence. Assume that $(f_1,\ldots,f_{k-1})\subseteq I,$ and that $f_{k+1},\ldots,f_n$ is an $R/I$--regular sequence. Then, Conjecture \ref{conjecture length Tor} holds.
\end{prop}

\begin{proof}
\ Let $g_1,\ldots,g_k$ be a regular sequence generating $I$. Moreover, set
\[
A:=\frac{R}{(g_1,\ldots,g_k,f_{k+1},\ldots,f_n)},
\]
which by our assumption is a complete intersection ring, in particular Gorenstein. Fix an integer $j\geq 1$, and denote by $\mathbf{0}_{k-1}$ the sequence of $(k-1)$ consecutive zeroes. Now, using \cite[Corollary 5.16]{Vasconcelosbookintegralclosure} one has that
\[
H_j (K_{\bullet}(f_1,\ldots,f_n;R/I))=H_j (K_{\bullet}(f_k,\mathbf{0}_{k-1};A)).
\]
Now, using \eqref{huneke trick} $k-1$ times one ends up with the following isomorphism:
\begin{equation}\label{huneke trick again}
H_j (K_{\bullet}(f_1,\ldots,f_n;R/I))=\bigoplus_{s=0}^j H_{j-s} (K_{\bullet}(f_k;A))^{\oplus\binom{k-1}{s}}.
\end{equation}
However, since the homology of the Koszul complex with respect to one element is concentrated always in degrees zero and one, the natural isomorphism \eqref{huneke trick again} boils down to
\[
H_j (K_{\bullet}(f_1,\ldots,f_n;R/I))=H_1 (K_{\bullet}(f_k;A))^{\oplus\binom{k-1}{j-1}}\oplus H_0 (K_{\bullet}(f_k;A))^{\oplus\binom{k-1}{j}}.
\]
Moreover, since $A$ is Gorenstein we have, using \cite[Theorem 2.2]{MillerRahmatiStriuliduality}, that
\[
\Hom_A (\Hom_A (H_1 (K_{\bullet}(f_k;A),A),A)\cong\Hom_A (H_0 (K_{\bullet}(f_k;A),A)
\]
and therefore, since the length of the linear dual of a module of finite length agree with the length of the module itself \cite[page 350]{dualityfinitelength}, we get that
\[
\ell\left(H_1 (K_{\bullet}(f_k;A))\right)=\ell\left(H_0 (K_{\bullet}(f_k;A))\right).
\]
In this way, combining these two last equalities with the additivity of the length and the Pascal identity
\[
\binom{k}{j}=\binom{k-1}{j-1}+\binom{k-1}{j}
\]
we obtain, taking into account
\[
\ell\left(H_0 (K_{\bullet}(f_k;A))\right)=\ell\left(\frac{R}{I+J}\right),
\]
that
\[
\ell(\Tor_j^R (R/I,R/J))=\binom{k}{j}\ell\left(\frac{R}{I+J}\right).
\]
This concludes the proof.
\end{proof}

The same strategy used along Proposition \ref{the conjecture in quasi good position} can be used to prove the following generalization of it.

\begin{prop}\label{the conjecture in 2-quasi good position}
Let $k\geq 3$, suppose that $I$ can be generated by a regular sequence of length $k$, and let $J=(f_1,\ldots,f_n),$ where $\mathbf{f}=f_1,\ldots,f_n$ is an $R$--regular sequence. Assume that $(f_1,\ldots,f_{k-2})\subseteq I,$ and that $f_{k+1},\ldots,f_n$ is an $R/I$--regular sequence. Then, Conjecture \ref{conjecture length Tor} holds.
\end{prop}

\begin{proof}
\ As before, let $g_1,\ldots,g_k$ be a regular sequence generating $I$, and set
\[
A:=\frac{R}{(g_1,\ldots,g_k,f_{k+1},\ldots,f_n)}.
\]
Fix $j\geq 1$ be an integer, and denote by $\mathbf{0}_{k-2}$ the sequence of $(k-2)$ consecutive zeroes. Now, using \cite[Corollary 5.16]{Vasconcelosbookintegralclosure} one has that
\[
H_j (K_{\bullet}(f_1,\ldots,f_n;R/I))=H_j (K_{\bullet}(f_{k-1},f_k,\mathbf{0}_{k-2};A)).
\]
Now, using \eqref{huneke trick} $k-2$ times one ends up with the following isomorphism:
\begin{equation}\label{huneke trick twice}
H_j (K_{\bullet}(f_1,\ldots,f_n;R/I))=\bigoplus_{s=0}^j H_{j-s} (K_{\bullet}(f_{k-1},f_k;A))^{\oplus\binom{k-2}{s}}.
\end{equation}
However, since the homology of the Koszul complex with respect to two elements is concentrated always in degrees zero, one and two, the natural isomorphism \eqref{huneke trick twice} boils down to
\begin{align*}
H_j (K_{\bullet}(f_1,\ldots,f_n;R/I))=& H_2 (K_{\bullet}(f_{k-1},f_k;A))^{\oplus\binom{k-2}{j-2}}\oplus H_1 (K_{\bullet}(f_{k-1},f_k;A))^{\oplus\binom{k-2}{j-1}}\\
& \oplus H_0 (K_{\bullet}(f_{k-1},f_k;A))^{\oplus\binom{k-2}{j}}.
\end{align*}
Moreover, since $A$ is Gorenstein we have, using \cite[Theorem 2.2]{MillerRahmatiStriuliduality}, that
\[
\Hom_A (\Hom_A (H_2 (K_{\bullet}(f_{k-1},f_k;A),A),A)\cong\Hom_A (H_0 (K_{\bullet}(f_{k-1},f_k;A),A),
\]
and therefore, using again that the length of the linear dual of a module of finite length agree with the length of the module itself, we have that
\[
\ell\left(H_2 (K_{\bullet}(f_{k-1},f_k;A))\right)=\ell\left(H_0 (K_{\bullet}(f_{k-1},f_k;A))\right).
\]
On the other hand, using again that $A$ is Gorenstein, we can repeat the calculation done for proving Conjecture \ref{conjecture length Tor} in the case of two elements to obtain the following equality:
\[
\ell\left(H_1 (K_{\bullet}(f_{k-1},f_k;A))\right)=2\ell\left(H_0 (K_{\bullet}(f_{k-1},f_k;A))\right)
\]
In this way, combining all these equalities with the additivity of the length and the identity
\begin{align*}
& \binom{k-2}{j-2}+2\binom{k-2}{j-1}+\binom{k-2}{j}=\binom{k-2}{j-2}+\binom{k-2}{j-1}+\binom{k-2}{j-1}+\binom{k-2}{j}\\
& =\binom{k-1}{j-1}+\binom{k-1}{j}=\binom{k}{j}
\end{align*}
we end up with
\[
\ell(\Tor_j^R (R/I,R/J))=\binom{k}{j}\ell\left(\frac{R}{I+J}\right),
\]
once again keeping in mind that
\[
\ell\left(H_0 (K_{\bullet}(f_{k-1},f_k;A))\right)=\ell\left(\frac{R}{I+J}\right).
\]
The proof is therefore completed.
\end{proof}
Our next goal is to provide a concrete example where one can effectively apply Proposition \ref{the conjecture in 2-quasi good position}. We will be back to the monomial case in Section \ref{the monomial case section}.

\begin{ex}\label{the monomial case with two freedom degrees}
Let $n\geq 4$ be an integer, let $3\leq k\leq n-1$ be an integer, let $(a_1,\ldots,a_k)\in\mathbb{N}^k,$ let $(b_1,\ldots, b_n)\in\mathbb{N}^n$ such that $a_j\leq b_j$ for any $1\leq j\leq k-2,$ let $S=\mathbb{K}[x_1,\ldots,x_n]$ be the polynomial ring in $n$ variables over a field $\mathbb{K},$ let $\mathfrak{m}=(x_1,\ldots,x_n),$ and set
\[
R:=S_{\mathfrak{m}},\ I:=(x_1^{a_1},\ldots,x_k^{a_k})R,\ J:=(x_1^{b_1},\ldots,x_n^{b_n})R.
\]
Notice that the following statements hold.

\begin{enumerate}[(i)]

\item\ $I$ is generated by an $R$--regular sequence of length $k.$

\item\ $J$ is generated by an $R$--regular sequence of length $n.$

\item\ Since, for any $1\leq j\leq k-2,$ $a_j\leq b_j,$ one has that $(x_1^{b_1},\ldots,x_{k-2}^{b_{k-2}})R\subseteq I.$

\item\ $x_{k+1}^{b_{k+1}},\ldots,x_n^{b_n}$ is an $R/I$--regular sequence of length $n-k\geq 1.$

\end{enumerate}
Therefore, by the foregoing statements we can apply Proposition \ref{the conjecture in 2-quasi good position} to conclude that, in this case,
\[
\ell(\Tor_i^R (R/I,R/J))=\binom{k}{i}\ell\left(\frac{R}{I+J}\right).
\]
\end{ex}
The same strategy used in all our previous results can also be applied to prove the following:

\begin{lm}\label{the most we can say}
Let $k\geq 3$, suppose that $I$ can be generated by a regular sequence of length $k$, and let $J=(f_1,\ldots,f_n),$ where $\mathbf{f}=f_1,\ldots,f_n$ is an $R$--regular sequence. Assume that $(f_1,\ldots,f_s)\subseteq I$ for some $1\leq s\leq k-2$ and that $f_{k+1},\ldots,f_n$ is an $R/I$--regular sequence. Then, one has that
\[
\ell\left(H_j (\mathbf{f};R/I)\right)=\sum_{i=0}^{k-s}\binom{s}{j-i}\ell\left(H_i (f_{s+1},\ldots ,f_k;R/(I+(f_{k+1},\ldots,f_n))\right).
\]
\end{lm}

\begin{rk}
Notice that Lemma \ref{the most we can say} is a key technical fact for reducing the conjecture to the case where our ideal $J$ is generated by less elements. The price we have to pay is that we have to calculate Koszul homologies not over a regular local ring, but over the complete intersection ring $R/I$. Moreover, in this case $J$ could not be generated by an $R/I$--regular sequence. 
\end{rk}

\section{The monomial case}\label{the monomial case section}

Here, we present another instance where Conjecture \ref{conjecture length Tor} holds; namely, in the monomial setting. More precisely, the aim of this section is to show that Conjecture \ref{conjecture length Tor} holds in case both ideals $I$ and $J$ are monomial ideals generated by pure powers of the variables in a polynomial ring.


\begin{nt}\label{notation in the monomial setting}
Hereafter in this section, let $\mathbb{K}$ be a field, let $R=\mathbb{K}[x_1,\ldots,x_n]$ be the polynomial ring in $n\geq 1$ variables over $\mathbb{K}$, let $(a_1,\ldots,a_k)\in\mathbb{N}^k$ (for some $1\leq k\leq n$) and let $(b_1,\ldots,b_n)\in\mathbb{N}^n.$ Assume also that $a_j\neq 0$ for all $1\leq j\leq k$ and that $b_l\neq 0$ for all $1\leq l\leq n.$ Finally, set $I:=(x_1^{a_1},\ldots,x_k^{a_k}),\ J:=(x_1^{b_1},\ldots,x_n^{b_n})$.
\end{nt}

Keeping in mind the previous notation, the main result of this section is the following:

\begin{teo}\label{the conjecture in the monomial case}
Preserving Notation \ref{notation in the monomial setting}, one has that
\[
\ell_R (\Tor_i^R (R/I,R/J))=\binom{k}{i}\ell_R\left(\frac{R}{I+J}\right).
\]
\end{teo}

\begin{proof}
\ First, we set the following notation.
\begin{align*}
& S_j:=\mathbb{K}[x_j],\ 1\leq j\leq n,\\
& M_j:=\begin{cases}
\dfrac{S_j}{(x_j^{a_j})},\text{ if }1\leq j\leq k,\\
S_j,\text{ if }k+1\leq j\leq n.
\end{cases}\\
& N_j:=
\dfrac{S_j}{(x_j^{b_j})},\ 1\leq j\leq n.
\end{align*}
Then, we have the following $\mathbb{K}$--vector space isomorphisms:
\[
\dfrac{R}{I}\cong\bigotimes_{j=1}^n M_j,\quad \dfrac{R}{J}\cong\bigotimes_{j=1}^n N_j.
\]
Now, fix $1\leq j\leq k$. We have the following short exact sequence of $S_j$--modules:
\[
\xymatrix{0\ar[r]& S_j\ar[r]^-{\cdot x_j^{a_j}}& S_j\ar[r]& M_j\ar[r]& 0.}
\]
Tensoring this short exact sequence with $N_j$ one obtains, for any integer $1\leq j\leq k$, the complex
\[
(C_j)_{\bullet}:\ \xymatrix{0\ar[r]& N_j\ar[r]^-{\cdot x_j^{a_j}}& N_j\ar[r]& 0.}
\]
On the other hand, for $k+1\leq j\leq n$, we define $(C_j)_{\bullet}$ as the complex
\[
\xymatrix{0\ar[r]& S_j\ar[r]^-{x_j^{b_j}}& S_j\ar[r]& 0.}
\]
Now, fix an integer $1\leq j\leq k$. The complex $(C_j)_{\bullet}$ has non-zero homology at degrees $0$ and $1$. Indeed, on the one hand the zeroth homology of this complex has length
\[
\ell\left(\dfrac{N_j}{x_j^{a_j}N_j}\right)=\min (a_j,b_j).
\]
On the other hand, its first homology has the same length, because $(0:_{N_j} x_j^{a_j})$ is spanned by the final $\min (a_j,b_j)$ monomials modulo $x_j^{b_j}$. Moreover, for $j>k$, the complex $(C_j)_{\bullet}$ has only homology at degree zero, with length $b_j$. Recall also that the full Koszul complex is the tensor product over $\mathbb{K}$ of these one variable complexes \cite[end of page 46]{BrunsHerzog1993}.

Now, we define the following complexes
\[
A_{\bullet}:=\bigotimes_{j=1}^k (C_j)_{\bullet},\quad B_{\bullet}:=\bigotimes_{j=k+1}^n (C_j)_{\bullet}.
\]
By the previous remarks, it is clear that, for any $0\leq i\leq n$,
\[
\Tor_i^R \left(\dfrac{R}{I},\dfrac{R}{J}\right)=H_i\left(A_{\bullet}\otimes B_{\bullet}\right).
\]
Now, by the Kunneth formula in the form displayed in \cite[Corollary 10.84]{Rotman2009}, we have a $\mathbb{K}$--vector space isomorphism
\begin{equation}\label{Kunneth formula: try 1}
\Tor_i^R \left(\dfrac{R}{I},\dfrac{R}{J}\right)\cong\bigoplus_{p+q=i}\left(H_p\left(A_{\bullet}\right)\otimes H_q (B_{\bullet})\right).
\end{equation}
Now, let us first focus on the homology of the complex $B_{\bullet}$. By combining the Kunneth formula with the definition of $B_{\bullet}$ we have that
\[
H_q(B_{\bullet})\cong\bigoplus_{s_{k+1}+\ldots+s_n=q}\left(\bigotimes_{j=k+1}^n H_{s_j}((C_j)_{\bullet})\right)=\begin{cases}
0,\text{ if }q\neq 0,\\
\bigotimes_{j=k+1}^n H_0\left((C_j)_{\bullet}\right),\text{ if }q=0.
\end{cases}
\]
Taking into account this isomorphism, we can rewrite \eqref{Kunneth formula: try 1} as follows.
\begin{equation}\label{Kunneth formula: try 2}
\Tor_i^R \left(\dfrac{R}{I},\dfrac{R}{J}\right)\cong H_i\left(A_{\bullet}\right)\otimes\left(\bigotimes_{j=k+1}^n H_0\left((C_j)_{\bullet}\right)\right).
\end{equation}
Now, we focus on the homology of $A_{\bullet}$. Since each complex $(C_j)_{\bullet}$ has only non--zero homologies at degrees zero and one, and both have the same length, we have the following isomorphism of $\mathbb{K}$--vector spaces, where by convention $\binom{k}{i}=0$ for $k>i$:
\[
H_i\left(A_{\bullet}\right)\cong\left(\bigotimes_{j=1}^k H_0 ((C_j)_{\bullet})\right)^{\oplus\binom{k}{i}}.
\]
In this way, by combining this last isomorphism with \eqref{Kunneth formula: try 2} and the foregoing calculations of lengths, we obtain that
\[
\ell\left(\Tor_i^R\left(\dfrac{R}{I},\dfrac{R}{J}\right)\right)=\ell (H_i (A_{\bullet}))\cdot\ell (H_0 (B_{\bullet}))=\binom{k}{i}\prod_{j=1}^k\min (a_j,b_j)\prod_{j=k+1}^n b_j.
\]
On the other hand, since
\[
I+J=(x_1^{\min(a_1,b_1)},\ldots,x_k^{\min (a_k,b_k)},x_{k+1}^{b_{k+1}},\ldots,x_n^{b_n}),
\]
hence
\[
\ell\left(\dfrac{R}{I+J}\right)\prod_{j=1}^k\min (a_j,b_j)\prod_{j=k+1}^n b_j.
\]
This concludes the proof.
\end{proof}

\section{Some counterexamples}\label{section: some counterexamples}

The goal of this section is to exhibit two counterexamples to Conjecture \ref{conjecture length Tor} showing that, in general, this conjecture does not always hold.

Our first example also shows that the conjecture is not true, in general, when either $I$ or $J$ is generated by a regular sequence given by binomials. The calculations in the following examples were done with Macaulay2 \cite{M2}.

\begin{ex}
Let $S=\mathbb{Q}[x,y,z]$ be a polynomial ring in three variables over the rational numbers, set $\mathfrak{m}=(x,y,z),$ and set $R:=S_{\mathfrak{m}},$ which is a regular local ring. Consider the ideals of $R$
\[
I=(x^2,y^2,z^2),\ J=(xy+x^2,xz+z^2,yz+y^2).
\]
In this case, one has that
\[
\ell\left(\Tor_1^R (R/I,R/J)\right)=14\neq 12=\binom{3}{1}\ell\left(\frac{R}{I+J}\right).
\]
\end{ex}

Our second example is the following one. In this example, there is defect not only in one Tor length, but in two.

\begin{ex}
Let $S=\mathbb{Q}[x,y,z,w]$ be a polynomial ring in four variables over the rational numbers, set $\mathfrak{m}=(x,y,z,w),$ and set $R:=S_{\mathfrak{m}},$ which is a regular local ring. Consider the ideals of $R$
\[
I=(x^2,y^2,z^2,w^2),\ J=\left(\frac{\partial h_3}{\partial x},\frac{\partial h_3}{\partial y},\frac{\partial h_3}{\partial z},\frac{\partial h_3}{\partial w}\right),
\]
where $h_3$ denotes the complete symmetric polynomial of degree $3.$ We know, thanks to \cite[Lemma 3.1]{Kumarregularsequences}, that $J$ is generated by an $R$--regular sequence. However, one has that
\begin{align*}
& \ell\left(\Tor_1^R (R/I,R/J)\right)=30\neq 28=\binom{4}{1}\ell\left(\frac{R}{I+J}\right),\\
& \ell\left(\Tor_2^R (R/I,R/J)\right)=46\neq 42=\binom{4}{2}\ell\left(\frac{R}{I+J}\right).
\end{align*}
\end{ex}

\section{Application to singularity theory}\label{section: singularity theory}

As we have already explained in the Introduction of this manuscript, this section can be read independently of the rest of the paper. Our motivation here is to explain why Conjecture \ref{conjecture length Tor} has some interest from the perspective of Singularity Theory. In this section, $\mathbb{C}\{x_1,\ldots,x_n\}$ will denote the ring of convergent power series with coefficients in the field of complex numbers.








Let $(X,0) \subset (\mathbb{C}^{n},0)$ be a reduced analytic variety, and let $f$ be a function germ. The Bruce--Roberts number of $f$ with respect to $(X,0)$, denoted $\mu_{\operatorname{BR}}(f,X)$, was defined in \cite{bruce1988critical} as the length of the quotient $\mathbb{C}\{x_1,\ldots,x_n\}/tf(\Theta_{X})$, where  
\[
\Theta_{X} := \{\xi \in \Theta_{n} \mid \xi(h) \in I_{X} \}
\]  
is the set of vector fields that are tangent to $(X,0)$, and $\Theta_n$ denotes the $\mathbb{C}\{x_1,\ldots,x_n\}$--module of germs of vector fields on $(\mathbb{C}^n,0)$. The notation $\Theta_{X}$ is also commonly written as $\operatorname{Der}{(-\log)}(X)$.
This number is a generalization of the Milnor number of a function germ.  If $X = \mathbb{C}^{n}$, then $\Theta_{X} = \Theta_{n}$, and in this case, $\mu_{\operatorname{BR}}(f,X) = \mu(f)$. In general, the Bruce--Roberts number of $f$ with respect to $(X,0)$ is greater than or equal to the Milnor number of $f$. 

The first relation between these numbers was presented in \cite[Theorem 3.1]{nuno2013bruce}, where $(X,0)$ is a weighted homogeneous hypersurface with an isolated singularity. In this case,  
\begin{equation}\label{Bruce roberts and milnor number for an isolated singularity}
\mu_{\operatorname{BR}}(f,X) = \mu(f) + \mu(X \cap f^{-1}(0)),
\end{equation}
where $\mu(X \cap f^{-1}(0))$ is the Milnor number of the ICIS $(X \cap f^{-1}(0),0)$, as defined in \cite{hamm1971lokale, hamm2006topology}.

In \cite[Corollary 4.1]{nuno2020bruce}, \eqref{Bruce roberts and milnor number for an isolated singularity} was extended to any isolated hypersurface singularity. In this case,  
\begin{equation}\label{Bruce roberts for any isolated hypersurface}
\mu_{\operatorname{BR}}(f,X) = \mu(f) + \mu(X \cap f^{-1}(0)) + \mu(X,0) - \tau(X,0),
\end{equation}  
where $\mu$ and $\tau$ denote respectively the Milnor and Tjurina numbers of $(X,0)$. \eqref{Bruce roberts for any isolated hypersurface} can be regarded as a natural extension of \eqref{Bruce roberts and milnor number for an isolated singularity}. In fact, an isolated hypersurface singularity is weighted homogeneous if and only if the equality $\mu(X,0) = \tau(X,0)$ holds (see \cite{saito1971quasihomogene}).

One class of varieties that extends the concept of an isolated hypersurface singularity is the class of isolated complete intersection singularity (ICIS). Considering any ICIS, it was raised the following conjecture (see \cite{conjectureTor}):

\begin{con}\label{BRconjecture}
Let $(X,0) \subset (\mathbb{C}^{n},0)$ be an ICIS of codimension $0 < k < n$, and let $f$ be a function germ such that $\mu_{\operatorname{BR}}(f,X) < \infty$. Then,  
\[
\mu_{\operatorname{BR}}(f,X) = \mu(f) + \mu(X \cap f^{-1}(0),0) + \mu(X,0) - \tau(X,0) + (k-1) \dim_{\mathbb{C}} \left(\frac{\mathbb{C} \{x_1, \ldots, x_n\}}{J(f) + I_X}\right).
\]
\end{con}
Conjecture \ref{BRconjecture} is an extension of the relation that satisfies an isolated hypersurface singularity. In \cite[Theorem 4.7]{conjectureTor}, it was proved that  
\begin{align*}
\mu_{\operatorname{BR}}(f,X) &= \mu(f) + \mu(X \cap f^{-1}(0),0) + \mu(X,0) - \tau(X,0) - \dim_{\mathbb{C}} \left(\frac{\mathbb{C} \{x_1, \ldots, x_n\}}{J(f) + I_{X}}\right) \\
& + \dim_{\mathbb{C}} \operatorname{Tor}_{1} \left( \frac{\mathbb{C} \{x_1, \ldots, x_n\}}{I_X}, \frac{\mathbb{C} \{x_1, \ldots, x_n\}}{J(f)} \right),
\end{align*}
for any ICIS $(X,0)$ of codimension $0 < k < n$, and any analytic function germ $f$ such that $\mu_{\operatorname{BR}}(f,X) < \infty$.

To conclude the proof of Conjecture \ref{BRconjecture}, it is necessary to prove the equality  
\[
\dim_{\mathbb{C}} \operatorname{Tor}_{1} \left( \frac{\mathbb{C} \{x_1, \ldots, x_n\}}{I_X}, \frac{\mathbb{C} \{x_1, \ldots, x_n\}}{J(f)} \right) = k \dim_{\mathbb{C}} \left(\frac{\mathbb{C} \{x_1, \ldots, x_n\}}{J(f) + I_{X}}\right).
\]

With this goal in mind,  Conjecture \ref{conjecture length Tor} was formulated in \cite{conjectureTor}. In section \ref{section: some counterexamples}, we have presented counterexamples to Conjecture \ref{conjecture length Tor}. Since in all of them the ideal $I$ is generated by a regular sequence of length $n$, we can now present a more restrictive conjecture that is still relevant in the context described in this section.

\begin{con}\label{conjecture length Tor2}
Let $(R,\mathfrak{m})$ be a regular local ring of dimension $n$, let $I$ be an ideal generated by an $R$--regular sequence of length $k \leq n-1$, and let $J$ be an ideal generated by a regular sequence of length $n$. Then, for any $0\leq i\leq k$, we have that
\[
\ell\left(\operatorname{Tor}_i^R (R/I, R/J)\right) = \binom{k}{i} \ell\left(\frac{R}{I+J}\right).
\]
\end{con}

\begin{rk}
Keeping in mind the content of our new restricted Conjecture \ref{conjecture length Tor2}, we want to single out the following facts.

\begin{enumerate}[(i)]

\item\ As noted in Section \ref{section: some positive results}, the case where the ideal $I$ is generated by a regular sequence of length $k=2$ has already been proven.

\item\ The case $k=1$ can be proved as follows. Set $A:=R/J$, which is an Artinian complete intersection ring, and let $f\in R$ be a generator of $I$. In this case, we have the following exact sequence:
\[
\xymatrix{0\ar[r]& (0:_A f)\ar[r]& A\ar[r]^-{\cdot f}& A\ar[r]& \dfrac{A}{fA}\cong\dfrac{R}{I+J}\ar[r]& 0.}
\]
From this short exact sequence we deduce that
\[
\ell(\Tor_1^R (R/I,R/J))=\ell ((0:_A f))=\ell\left(\dfrac{A}{fA}\right)=\ell\left(\dfrac{R}{I+J}\right),
\]
which is what we want to show.

\item\ As we have already seen along the paper (see Section \ref{section: some positive results}), by imposing certain conditions on the ideal $I$ we have established Conjecture \ref{conjecture length Tor2} in some particular cases.

\end{enumerate}
\end{rk}

\section{Raising more general conjectures}\label{more general conjecture section}

In this paper, we see that Conjecture \ref{conjecture length Tor} is in general false. However, in all our counterexamples we always have the following lower inequality:
\[
\ell\left(\Tor_i^R (R/I,R/J)\right)\geq \binom{k}{i}\ell\left(\frac{R}{I+J}\right).
\]
This motivates us to formulate the following conjecture, that might be regarded as a variant of the Buchsbaum--Eisenbud--Horrocks problem in the local case.

\begin{quo}\label{generalized BEH: try two}
Let $(R,\mathfrak{m})$ be a regular local ring of dimension $n,$ let $I$ be an ideal of $R$ that contains a regular sequence of length $k\leq n,$ and let $J$ be an ideal of $R$ that contains a regular sequence of length $n$ such that $I+J$ is $\mathfrak{m}$--primary. Let $0\leq i\leq k$. Is it true that
\[
\ell\left(\Tor_i^R (R/I,R/J)\right)\geq\binom{k}{i}\ell\left(\frac{R}{I+J}\right)?
\]
\end{quo}

We want to observe that, in Question \ref{generalized BEH: try two} we cannot weaken the regular sequence containment assumption on $J$, as the next example shows.

\begin{ex}\label{outside ci not true in general}
Let $S=\mathbb{Q}[x,y,z,w],$ let $\mathfrak{m}=(x,y,z,w),$ set $R:=S_{\mathfrak{m}},$ and
\[
I:=(x^3,y^3,z^3,w^3)R,\ J:=(x(x+y+z),y(x+y+w),z(x+z+w),w(y+z+w))R.
\]
One can check that $\sqrt{I+J}=(x,y,z,w)R,$ that $I$ is a complete intersection, and that $J$ is an almost complete intersection (it has grade $3$ and $4$ minimal generators). In this case,
\begin{align*}
& \ell(\Tor_1 (R/I,R/J))=42<48=\binom{4}{1}\cdot 12=\binom{4}{1}\cdot\ell\left(\frac{R}{I+J}\right),\\
& \ell(\Tor_2 (R/I,R/J))=51<72=\binom{4}{2}\cdot 12=\binom{4}{2}\cdot\ell\left(\frac{R}{I+J}\right).
\end{align*}
\end{ex}

The next example may be regarded as a piece of evidence for Question \ref{generalized BEH: try two}.

\begin{ex}
Coming back to Example \ref{outside ci not true in general}, setting now
\[
J:=(x^3,y^3,z^3,w^3)R,\ I:=(x(x+y+z),y(x+y+w),z(x+z+w),w(y+z+w))R.
\]
and keeping in mind Question \ref{generalized BEH: try two}, we have
\begin{align*}
& \ell(\Tor_1 (R/I,R/J)))=42>36=\binom{3}{1}\cdot 12=\binom{3}{1}\cdot\ell\left(\frac{R}{I+J}\right),\\
& \ell(\Tor_2 (R/I,R/J)))=51>36=\binom{3}{2}\cdot 12=\binom{3}{2}\cdot\ell\left(\frac{R}{I+J}\right),\\
& \ell(\Tor_3 (R/I,R/J)))=25>12=\binom{3}{3}\cdot 12=\binom{3}{3}\cdot\ell\left(\frac{R}{I+J}\right).
\end{align*}
\end{ex}
We can also formulate a local variant of the so--called Total Rank Conjecture.

\begin{quo}\label{generalized TRC}
Let $(R,\mathfrak{m})$ be a regular local ring of dimension $n,$ let $I$ be an ideal of $R$ that contains a regular sequence of length $k\leq n,$ and let $J$ be an ideal of $R$ that contains a regular sequence of length $n$ such that $I+J$ is $\mathfrak{m}$--primary. Is it true that
\[
\sum_{i=0}^{n}\ell\left(\Tor_i^R (R/I,R/J)\right)\geq 2^k\ell\left(\frac{R}{I+J}\right)?
\]
\end{quo}
Notice that Question \ref{generalized TRC} is not covered by the results obtained by VandeBogert and Walker in \cite{WalkerTRCeven} because the tensor product of the Koszul complexes of two complete intersection ideals is not a priori concentrated in degrees $[0,\dim (R)+2]$.

\section*{Acknowledgements}
This project started when Alberto F. Boix and Barbara K. Lima--Pereira attended the IberoSing International Workshop 2022 celebrated at the Universidad Complutense de Madrid (Spain). We want to thank the organizers of the event, Pedro Gonz\'alez P\'erez, Patricio Almir\'on, Pablo Portilla and Juan Viu--Sos. We also want to thank Srikanth Iyengar and Mark Walker for their interest on this project, and for suggesting some of the examples we included in Section \ref{section: some counterexamples}.  We would also like to thank the referee for several useful comments and remarks that have substantially improved the contents and the correctness of this manuscript. Alberto F. Boix received partial support by grant PID2022-137283NB-C22 funded by  MICIU/AEI/10.13039/501100011033. Bárbara K. Lima-Pereira was supported by FAPESP grants 2022/08662-1 and 2023/04460-8.

\bibliographystyle{alpha}
\bibliography{AFBoixReferences}

\end{document}